\documentclass[11pt,a4paper]{amsart}
\usepackage{graphicx}
\usepackage{amsmath,amsfonts}
\usepackage{amsthm,amssymb,latexsym}
\usepackage[active]{srcltx}
\usepackage[usenames]{color}
\usepackage{tikz}
\usetikzlibrary{matrix,decorations.pathreplacing,calc}

\newcommand{\bfs}{\boldsymbol}

\newtheorem{theorem}{Theorem}[section]
\newtheorem{corollary}[theorem]{Corollary}
\newtheorem{lemma}[theorem]{Lemma}

\newtheorem*{fact*}{Fact}
\newtheorem{proposition}[theorem]{Proposition}
\theoremstyle{definition}

\newtheorem*{claim}{Claim}

\numberwithin{equation}{section}

\newcommand{\Z}{\mathbb Z}

\newcommand{\C}{\mathbb C}

\newcommand{\A}{\mathbb A}
\newcommand{\F}{\mathbb F}
\newcommand{\K}{\mathbb K}

\newcommand{\fq}{\F_{\hskip-0.7mm q}}

\newcommand{\cfq}{\overline{\F}_{\hskip-0.7mm q}}

\def\ifm#1#2{\relax \ifmmode#1\else#2\fi}

\newcommand{\klk}    {\ifm {,\ldots,} {$,\ldots,$}}

\begin{document}

\title[Average--case complexity of the Euclidean Algorithm]{Average--case complexity of the
Euclidean algorithm with a fixed polynomial over a finite field}
\author[N. Gim\'enez]{
Nardo Gim\'enez${}^{1}$}
\author[G. Matera]{
Guillermo Matera${}^{1,2,3}$}
\author[M. P\'erez]{
Mariana P\'erez${}^{3,4}$}
\author[M. Privitelli]{
Melina Privitelli${}^{3,5}$}

\address{${}^{1}$
Universidad Nacional de Gene\-ral Sarmiento, Instituto del
Desarrollo Humano, J.M. Guti\'errez 1150 (B1613GSX) Los Polvorines,
Buenos Aires, Argentina} \email{agimenez@ungs.edu.ar,
gmatera@ungs.edu.ar}
\address{${}^{2}$Universidad de Buenos Aires, Facultad de Ciencias
Exactas y Naturales, Departamento de Matem\'atica, Ciudad
Universitaria, Pabell\'on I (1428) Buenos Aires, Argentina}
\address{${}^{3}$ Consejo Nacional de Investigaciones
Cient\'ificas y T\'ecnicas (CONICET),
Ar\-gentina}
\address{${}^{4}$
Universidad Nacional de Hurlingham, Instituto de Tecnolog\'ia e
Ingenier\'ia, Av. Gdor. Vergara 2222 (B1688GEZ), Villa Tesei, Buenos
Aires, Argentina} \email{mariana.perez@unahur.edu.ar}
\address{${}^{5}$Universidad Nacional de Gene\-ral Sarmiento, Instituto de Ciencias,
J.M. Guti\'errez 1150
(B1613GSX) Los Polvorines, Buenos Aires, Argentina}
\email{mprivite@ungs.edu.ar}

\thanks{The authors were partially supported by the grants
PIP CONICET 11220130100598 and PIO CONICET-UNGS 14420140100027}


\date{\today}%

\begin{abstract}
We analyze the behavior of the Euclidean algorithm applied to pairs
$(g,f)$ of univariate nonconstant polynomials over a finite field
$\fq$ of $q$ elements when the highest-degree polynomial $g$ is
fixed. Considering all the elements $f$ of fixed degree, we
establish asymptotically optimal bounds in terms of $q$ for the
number of elements $f$ which are relatively prime with $g$ and for
the average degree of $\gcd(g,f)$. The accuracy of our estimates is
confirmed by practical experiments. We also exhibit asymptotically
optimal bounds for the average-case complexity of the Euclidean
algorithm applied to pairs $(g,f)$ as above.
\end{abstract}

\maketitle

%
%
\section{Introduction}
Let $\fq$ be the finite field of $q$ elements, where $q$ is a prime
power, let $T$ be an indeterminate over $\fq$ and $\fq[T]$ the ring
of univariate polynomials in $T$ with coefficients in $\fq$. In this
paper we are concerned with the polynomial gcd problem for elements
of $\fq[T]$, namely the problem of computing the greatest common
divisor of two nonzero polynomials in $\fq[T]$.

The fundamental computational tool for this problem is the Euclidean
algorithm, and many variants of it are known in the literature (see,
e.g., \cite{GaGe99}). It is well-known that the Euclidean algorithm
in $\fq[T]$ requires a number of polynomial divisions which is
linear in the degree of the input polynomials. In particular, we are
interested in its average-case complexity, which has been the
subject of several papers. The paper \cite{MaGa90} establishes the
average-case complexity of the Euclidean algorithm and some variants
of it based on explicit counting. In \cite{Norton89}, the
average-case complexity of variants of the Euclidean algorithm is
considered using generating functions. Finally, \cite{LhVa08} and
\cite{BeNaNaVa14} analyze the average-case complexity and related
costs of the Euclidean algorithm and variants using tools of
analytic combinatorics such as bivariate generating functions.

All these results consider the average, for fixed degrees $e>d>0$,
over the set of pairs $(g,f)\in\fq[T]\times\fq[T]$ with $g$ monic of
degree $e$ and $f$ either of degree at most $d$, or of degree less
than $e$, assuming the uniform distribution of pairs. Nevertheless,
there are important tasks which rely heavily on the computation of
gcd's and lie outside the scope of these analyses. For example, a
critical step in the standard algorithm for finding the roots in
$\fq$ of a polynomial $f\in\fq[T]$ with $\deg f<q$ consists of
computing $\gcd(T^q-T,f)$ (see, e.g., \cite{GaGe99}). As the first
element in the pair $(T^q-T,f)$ is a fixed polynomial, average-case
analyses as before do not contribute to the analysis of the
complexity of this problem.

In this paper we consider, for fixed degrees $e>d>0$ and a fixed
(arbitrary) $g\in\fq[T]$ monic of degree $e$, the average-case
complexity of the Euclidean algorithm over the set of pairs $(g,f)$
with $f\in\fq[T]$ monic of degree $d$, endowed with the uniform
probability. We shall be interested in the case $q\gg e$; in this
sense, all our results may be regarded as asymptotic in $q$.

We discuss a number of issues concerning this case of the Euclidean
algorithm. Our first result shows that the average degree
$E[\mathcal{X}_g]$ of $\gcd(g,f)$ for a random element $f$ of
$\fq[T]$, monic of degree $d$, decreases fast as $q$ tends to
infinity. Further, we prove that the decrease rate depends on the
factorization pattern of $g$ 
(see Theorem \ref{teo: average degree mcd} for a precise statement):
\begin{equation}\label{eq: average degree - intro}
\bigg|E[\mathcal{X}_g]-\frac{k\lambda_k^*}{q^k}\bigg|=
\mathcal{O}\bigg(\frac{1}{q^{k+1}}\bigg),
\end{equation}
where $\lambda_k^*$ denotes the number of distinct monic irreducible
factors of $g$ of degree $k$. The average degree of the gcd of a
random pair of elements in $\fq[T]$ of degrees $e$ and $d$ as above
is $(1-q^{-d})/(q-1)$ (see \cite[Corollary 2.6]{MaGa90}). Our
result, although not as precise as the latter, confirms that in our
case the average degree of the gcd is $\mathcal{O}(q^{-1})$ (for
fixed $d,e$).

We also show that, with high probability, $g$ and a random monic
polynomial $f$ of $\fq[T]$ of degree $d$ are relatively prime. In
fact, we have the following estimate for the probability
$\mathcal{P}_0$ that $\gcd(g,f)=1$ (see Theorem \ref{teo: prob
coprime}):
\begin{equation}\label{eq: average prob rel prime - intro}
\bigg|\mathcal{P}_0-\bigg(1 - \frac{\lambda_k^*}{q^k}\bigg)\bigg|=
\mathcal{O}\bigg(\frac{1}{q^{k+1}}\bigg).
\end{equation}
This may be compared with the probability $1-1/q$ that a random pair
of elements of $\fq[T]$ of degrees $e$ and $d$ are relatively prime
(see, e.g., \cite[Proposition 2.4]{MaGa90}).

Finally, we analyze the average number $E[t_g^{\sf div}]$,
$E[t_g^{\div}]$ and $E[t_g^{-, \times}]$ of polynomial divisions,
divisions in $\fq$, and additions/multiplications in $\fq$,
performed by the Euclidean algorithm. We have the following bounds
(see Theorem \ref{teo: average cost mcd}):
\begin{align*}
\bigg|\frac{E[t^{\sf div}_g]}{d+1}-1\bigg| \le\frac{de}{q},\qquad
\bigg|\frac{E[t^{\div}_g]}{e+d+1}-1\bigg|\le \frac{de}{q}, \qquad
\bigg|\frac{E[t^{-, \times}_g]}{de}-1\bigg|\le \frac{de}{q}.
\end{align*}
The main terms in these bounds agree with those in the corresponding
ones for random pairs of polynomials of degree $e$ and $d$ with
$e>d$, according to \cite[Theorem 2.1]{MaGa90}.

Our approach relies on estimating the number of polynomials $f$ for
which $\gcd(g,f)$ has a given degree. For this purpose we use
classical tools of elimination theory, which are combined with
bounds on the number of common zeros with coordinates in $\fq$ of
multivariate polynomials defined over the algebraic closure $\cfq$
of $\fq$. Another critical point is a lower bound on the number of
polynomials $f$ for which the Euclidean algorithm performs the
highest possible number of steps. Such a lower bound relies on a
description of certain coefficients of the sequence of quotients and
remainders determined by the Euclidean algorithm in terms symmetric
functions, following \cite{Lascoux03}. Combining such a description
with upper bounds on the number of zeros with coordinates in $\fq$
of multivariate polynomials with coefficients in $\fq$, we are able
to control the number of polynomials $f$ for which $\gcd(g,f)$ has a
given degree. Our results are then expressed in terms of these
quantities.


The paper is organized as follows. In Section \ref{sec: basic
notions and notations} we recall the description of remainders and
quotients arising in the Euclidean algorithm applied to a
``generic'' pair of polynomials of given degrees in terms of
symmetric functions. In Section \ref{sec: degree bounds} we use this
machinery to estimate the degrees of the leading coefficients of the
remainders in the generic case and we consider the behavior of the
Euclidean algorithm under specializations. In Section \ref{sec:
average degree} we estimate on the number of polynomials $f$ for
which $\gcd(g,f)$ has a given degree, which are used to prove
\eqref{eq: average degree - intro} and \eqref{eq: average prob rel
prime - intro}. In Section \ref{sec: average case analysis} we use
the results of Sections \ref{sec: degree bounds} and \ref{sec:
average degree} to establish the results on the average-case
complexity. Finally, in Section \ref{section: simulations} we report
on some simulations we perform which show the behavior predicted by
the theoretical estimates \eqref{eq: average degree - intro} and
\eqref{eq: average prob rel prime - intro}.
%
%
\section{Basic notions and notations}
\label{sec: basic notions and notations}
%
%
Let $\fq$ be the finite field of $q$ elements and $\cfq$ its
algebraic closure. Let $X_1,\ldots,X_n$ be indeterminates over
$\cfq$. For $\K=\fq$ or $\K=\cfq$, we denote by $\K[X_1,\ldots,X_n]$
the ring of multivariate polynomials in $X_1,\ldots,X_n$ and
coefficients in $\K$. 
By $\A^n$ we denote the affine $n$--dimensional space
$\A^n:=\cfq{\!}^n$, endowed with its Zariski topology over $\K$, for
which a closed set is the zero locus of a set of polynomials of
$\K[X_1,\ldots, X_n]$.
%
A subset $V\subset \A^r$ is an {\em affine variety of
defined over} $\K$ (or an affine $\K$--variety) if it is the set of
common zeros in $\A^n$ of polynomials $F_1,\ldots, F_{m} \in
\K[X_1,\ldots, X_n]$. 
We shall denote by $\{F_1=0,\ldots, F_m=0\}$ or $V(F_1,\ldots,F_m)$
the affine $\K$--variety consisting of the common zeros of
$F_1,\ldots, F_m$.
%
%

A $\K$--variety $V$ is {\em irreducible} if it cannot be expressed
as a finite union of proper $\K$--subvarieties of $V$.
%
Any $\K$--variety $V$ can be expressed as an irredundant union
$V=\mathcal{C}_1\cup \cdots\cup\mathcal{C}_s$ of irreducible
$\K$--varieties, unique up to reordering, called the {\em
irreducible} 
$\K$--{\em components} of $V$.
%
We say that $V$ has {\em pure dimension} $r$ if every irreducible
$\K$--component of $V$ has dimension $r$. A $\K$--variety of 
$\A^n$ of pure dimension $n-1$ is called a $\K$--{\em hypersurface}.
A $\K$--hypersurface of 
$\A^n$ can also be described as the set of zeros of a single nonzero
polynomial of $\K[X_1,\ldots, X_n]$.

The {\em degree} $\deg V$ of an irreducible $\K$--variety $V$ is the
maximum of the cardinality $|V\cap L|$ of $V\cap L$, considering all
the linear spaces $L$ of codimension $\dim V$ such that $|V\cap
L|<\infty$. More generally, following \cite{Heintz83} (see also
\cite{Fulton84}), if $V=\mathcal{C}_1\cup\cdots\cup \mathcal{C}_s$
is the decomposition of $V$ into irreducible $\K$--components, we
define the degree of $V$ as
$$\deg V:=\sum_{i=1}^s\deg \mathcal{C}_i.$$
The degree of a $\K$--hypersurface $V$ is the degree of a polynomial
of minimal degree defining $V$. In particular, the degree of a linear
variety is equal to 1. 
Let $\A^n(\fq)$ be the $n$--dimensional $\fq$--vector space $\fq^n$.
For
an affine variety $V\subset\A^n$, 
the set of $\fq$--rational points $V(\fq)$ of $V$ is defined as 
$V(\fq):=V\cap\A^n(\fq)$.
For an affine variety $V\subset\A^n$ of dimension $r$ and degree
$d\ge 0$, we have the following bound (see, e.g., \cite[Lemma
2.1]{CaMa06}):
\begin{equation}\label{eq: upper bound -- affine gral}
   |V(\fq)|\leq d\, q^r.
\end{equation}

\subsection{Symmetric functions and the Euclidean algorithm}
Next we gather the terminology and results we will use concerning
the description of the Euclidean algorithm in terms of symmetric
functions, following \cite{Lascoux03}.

We call a finite set of indeterminates $\mathbb{A}$ over $\fq$ an
\emph{alphabet} and denote its cardinality by $|\mathbb{A}|$. The
\textit{elementary symmetric  functions}  $\Lambda^i(\mathbb{A})$
and the \textit{complete functions} $S^i(\mathbb{A})$ $(i\geq 0)$
are defined by means of the following identities of formal power
series in the variable $z$:
$$
\prod_{a \in \mathbb{A}}(1+za)=\sum_{i\geq 0} \Lambda^i(\mathbb{A})z^i,
\quad \prod_{a \in \mathbb{A}}\frac{1}{1-za}=\sum_{i\geq 0} S^i(\mathbb{A})z^i.
$$
We further define $\Lambda^i(\mathbb{A}):=0$ and
$S^i(\mathbb{A}):=0$ for $i<0$. Observe that
$\Lambda^i(\mathbb{A})=0$ if $i> |\mathbb{A}|$. Writing
$\mathbb{A}+\mathbb{B}$ for the disjoint union of two alphabets
$\mathbb{A}$ and $\mathbb{B}$, we have the following \textit{Cauchy
formulas}:
\begin{gather}\label{eq: cauchy_formula}
\Lambda^i(\mathbb{A}+\mathbb{B})=\sum_{j+k=i}\Lambda^j(\mathbb{A})\Lambda^k(\mathbb{B}),
\quad
S^i(\mathbb{A}+\mathbb{B})=\sum_{j+k=i}S^j(\mathbb{A})S^k(\mathbb{B}),
\end{gather}
Define $S^i(\mathbb{A}-\mathbb{B})$ $(i\geq 0)$ by means of the
identity
$$
\frac{\prod_{b \in \mathbb{B}} (1-zb)}{\prod_{a\in \mathbb{A}}
(1-za)}= \sum_{i\geq 0} S^i(\mathbb{A}-\mathbb{B})z^i,
$$
and set $S^i(\mathbb{A}-\mathbb{B}):=0$ for $i < 0$. Define
$S^i(-\mathbb{A}):=(-1)^i\Lambda^i(\mathbb{A})$ for any integer $i$.
Thus, besides \eqref{eq: cauchy_formula} we have
\begin{equation}\label{eq: cauchy_formula_difference}
S^i(\mathbb{A}-\mathbb{B})=\sum_{j+k=i}S^j(\mathbb{A})S^k(-\mathbb{B}).
\end{equation}
We shall express polynomials using this terminology. Indeed, let
$n:=|\mathbb{A}|$ and identify a single indeterminate $T$ with the
alphabet $\{T\}$. Since $S^i(T)=T^i$ for $i\geq 0$, according to
\eqref{eq: cauchy_formula_difference} we have that
$$
S^{n}(T-\mathbb{A})=\sum_{i=0}^{n}S^{n-i}(-\mathbb{A})T^i
$$
is the polynomial in $T$ having $\mathbb{A}$ as its set of roots.
%
%
\subsection{Schur functions}
Let $\mathbb{A}$ be an alphabet of cardinality $n$. Given $J:=(j_1,
\dots, j_n)\in \mathbb{Z}_{\geq 0}^n$, the \emph{Schur function}
$S_J(\mathbb{A})$ is defined as the determinant
$$
S_J(\mathbb{A}):=\det \big(S^{j_k+k-h}(\mathbb{A})\big)_{1\le h,k\le
n}.
$$
In other words,
$$
S_J(\mathbb{A})= \det \left(
 \begin{array}{cccc}
 S^{j_1}(\mathbb{A}) & S^{j_2+1}(\mathbb{A}) &\!\!\! \cdots
 &\!\!\! S^{j_n+n-1}(\mathbb{A}) \\
 S^{j_1-1}(\mathbb{A}) & S^{j_2}(\mathbb{A}) & \ddots & \vdots \\
 \vdots & \vdots & \cdots &  S^{j_n+1}(\mathbb{A})\\
\!\!\! S^{j_1-n+1}(\mathbb{A})\!\!\! & S^{j_2-n+2}(\mathbb{A}) & \cdots & S^{j_n}(\mathbb{A})\\
 \end{array}
 \right).
$$
Given $n$, two sets of alphabets $\{\mathbb{A}_1, \dots,
\mathbb{A}_n\}$, $\{\mathbb{B}_1, \dots, \mathbb{B}_n\}$ and $J\in
\mathbb{Z}_{\geq 0}^n$, we have the \emph{multi-Schur function}
$$
S_J(\mathbb{A}_1-\mathbb{B}_1, \dots,
\mathbb{A}_n-\mathbb{B}_n):=\det
\big(S^{j_k+k-h}(\mathbb{A}_k-\mathbb{B}_k)\big)_{1\le h,k\le n},
$$
namely
\begin{align*}
&S_J(\mathbb{A}_1-\mathbb{B}_1, \dots, \mathbb{A}_n-\mathbb{B}_n)=\\
&\qquad\qquad\det \left(
 \begin{array}{cccc}
 S^{j_1}(\mathbb{A}_1-\mathbb{B}_1) & S^{j_2+1}(\mathbb{A}_2-\mathbb{B}_2) &\!\!\! \cdots
 &\!\!\! S^{j_n+n-1}(\mathbb{A}_n-\mathbb{B}_n) \\
 S^{j_1-1}(\mathbb{A}_1-\mathbb{B}_1) & S^{j_2}(\mathbb{A}_2-\mathbb{B}_2) & \ddots & \vdots \\
 \vdots & \vdots & \cdots &  S^{j_n+1}(\mathbb{A}_n-\mathbb{B}_n)\\
\!\!\! S^{j_1-n+1}(\mathbb{A}_1-\mathbb{B}_1)\!\!\! &
S^{j_2-n+2}(\mathbb{A}_2-\mathbb{B}_2)
& \cdots & S^{j_n}(\mathbb{A}_n-\mathbb{B}_n)\\
 \end{array}
 \right).
\end{align*}
Finally, given $H\in \mathbb{Z}_{\geq 0}^p$ and $K\in
\mathbb{Z}_{\geq 0}^q$, and alphabets $\mathbb{A}$, $\mathbb{B}$,
$\mathbb{C}$, $\mathbb{D}$, we shall consider the multi-Schur
function $S_{H;K}(\mathbb{A}-\mathbb{B}; \mathbb{C}-\mathbb{D})$
with index $(H;K)$, the concatenation of $H$ and $K$, and alphabets
$\mathbb{A}_1=\mathbb{A}, \dots, \mathbb{A}_p=\mathbb{A}$,
$\mathbb{B}_1=\mathbb{B}, \dots, \mathbb{B}_p=\mathbb{B}$,
$\mathbb{A}_{p+1}=\mathbb{C}, \dots, \mathbb{A}_{p+q}=\mathbb{C}$,
$\mathbb{B}_{p+1}=\mathbb{D}, \dots, \mathbb{B}_{p+q}=\mathbb{D}$.
When a tuple $J=(j_1, \dots, j_n)\in \mathbb{Z}_{\geq 0}^n$ with
$j_i=m$ for $1\le i \le n$ appears as an index in a Schur function,
we denote it as $m^n$. For example, for $\ell\in \mathbb{Z}_{\geq
0}$, we have
$$
S_{m^n;\, \ell}(\mathbb{A}-\mathbb{B}; T)= \det \left(
 \begin{array}{cccc}
 S^{m}(\mathbb{A}-\mathbb{B})  &\!\!\! \cdots
 &\!\!\! S^{m+n-1}(\mathbb{A}-\mathbb{B}) & T^{\ell +n} \\
 S^{m-1}(\mathbb{A}-\mathbb{B}) &  \ddots & \vdots & \vdots \\
\!\!\! \vdots\!\!\! &  \ddots & S^{m}(\mathbb{A}-\mathbb{B}) & T^{\ell+1}\\
\!\!\! S^{m-n}(\mathbb{A}-\mathbb{B})\!\!\! &  \cdots &
S^{m-1}(\mathbb{A}-\mathbb{B}) & T^{\ell}
 \end{array}
 \right).
$$
We have the following result (see \cite[equation
(1.4.8)]{Lascoux03}).
\begin{lemma}\label{lemma: Lascoux 148}
Let $J\in \mathbb{Z}_{\geq0}^n$, $k\in \mathbb{Z}_{\geq0}$,
$\mathbb{A}$, $\mathbb{B}$ alphabets and $T$ an indeterminate. Then
 $$
 S_{J}(\mathbb{A}-\mathbb{B}-T)T^k=S_{J;\, k}(\mathbb{A}-\mathbb{B};T).
 $$
\end{lemma}
%
%
\subsection{Remainders as Schur functions}
\label{subsec: remainder as Schur functions}
Let $\mathbb{A}$ and $\mathbb{B}$ be two alphabets of cardinalities
$e$ and $d$ respectively, with $e> d$. When the Euclidean algorithm
is applied to two generic polynomials $S^{e}(T-\mathbb{A})$ and
$S^{d}(T-\mathbb{B})$ we obtain $d$ remainders $\mathcal{R}_1,
\dots, \mathcal{R}_{d}$ and quotients $q_1, \dots, q_{d}$ satisfying
the following identities:
\begin{align}\label{eq: euclidean_alg}
S^{e}(T-\mathbb{A}) & = q_1S^{d}(T-\mathbb{B}) + \mathcal{R}_1, \\
         S^{d}(T-\mathbb{B}) & =  q_2 \mathcal{R}_1 + \mathcal{R}_2, \nonumber \\
         \mathcal{R}_1 & = q_3\mathcal{R}_2+ \mathcal{R}_3, \nonumber  \\
         & \ \ \vdots \nonumber   \\
         \mathcal{R}_{d-2}& = q_{d}\mathcal{R}_{d-1}+ \mathcal{R}_{d}. \nonumber
 \end{align}
Here $\deg_T q_1=e-d$, $\deg_T q_i=1$ for $2\le i \le d$, and
$\deg_T\mathcal{R}_i=d-i$ for $1\le i \le d$. It turns out that all
the remainders $\mathcal{R}_i$ are elements of the ring
$\fq[\mathbb{A}, \mathbb{B}][T]$. Further, we may express these
remainders in terms of Schur functions \cite[equation
(3.1.5)]{Lascoux03}:
\begin{equation}\label{eq: schur_remainder_formula}
\mathcal{R}_k=(-1)^{d-k+1}S_{(e-d+k)^{k-1}}(\mathbb{B}
-\mathbb{A}-T)S^{e}(T-\mathbb{A}) +
S_{k^{e-d+k-1}}(\mathbb{A}-\mathbb{B}-T)S^{d}(T-\mathbb{B}).
\end{equation}
%
%
\section{Degree bounds for the remainders in the generic case}
\label{sec: degree bounds}
In the sequel, for $1\le k \le d$ we denote by $F_k\in
\fq[\mathbb{A}, \mathbb{B}]$ the leading coefficient of
$\mathcal{R}_k$, considered as an element of $\fq[\mathbb{A},
\mathbb{B}][T]$. Let $\mathbf S :=(S^{i}(-\mathbb{B}): \, 1\le i \le
d)$ and $\mathbf T:=(S^i(\mathbb{A}): \, 1\le i \le e)$. Observe
that both $\mathbf S$ and $\mathbf T$ are algebraically independent
sets over $\fq$.
\begin{proposition}\label{prop: main}
$F_k$ is a nonzero element of $\fq[\mathbf T][\mathbf S]$ of degree
$\deg_{\mathbf S}F_k=e-d+k$. Further, it is monic  of degree $e-d+k$
in $S^{k}(-\mathbb{B})$.
\end{proposition}

\begin{proof}
Since $S^{d}(T-\mathbb{B})=\sum_{j=0}^{d}S^{d-j}(-\mathbb{B})T^j$,
we can write
$$
S_{k^{e-d+k-1}}(\mathbb{A}-\mathbb{B}-T)S^{d}(T-\mathbb{B})=
\sum_{j=0}^{d}S^{d-j}(-\mathbb{B})
S_{k^{e-d+k-1}}(\mathbb{A}-\mathbb{B}-T)T^j.
$$
By Lemma \ref{lemma: Lascoux 148} we have
$$
S_{k^{e-d+k-1}}(\mathbb{A}-\mathbb{B}-T)T^j =S_{k^{e-d+k-1};\,
j}(\mathbb{A}-\mathbb{B}; T),
$$
where
\begin{equation}\label{eq: prop: main: determinat_1}
S_{k^{e-d+k-1};\, j}(\mathbb{A}-\mathbb{B}; T)= \det \left(
 \begin{array}{cccc}
 S^k(\mathbb{A-\mathbb{B}}) & \cdots &\!\!\! S^{e-d+2k-2}(\mathbb{A-\mathbb{B}})
 &\!\!\! T^{j+e-d+k-1} \\
 S^{k-1}(\mathbb{A-\mathbb{B}}) & \ddots & \vdots & \vdots \\
 \vdots & \ddots & S^k(\mathbb{A-\mathbb{B}}) & T^{j+1} \\
\!\!\! S^{-(e-d-1)}(\mathbb{A-\mathbb{B}})\!\!\! & \cdots & S^{k-1}(\mathbb{A-\mathbb{B}}) & T^j \\
 \end{array}
 \right).
\end{equation}
Similarly, taking into account that
$S^{e}(T-\mathbb{A})=\sum_{h=0}^{e}S^{e-h}(-\mathbb{A})T^h$, we see
that
$$
S_{(e-d+k)^{k-1}}(\mathbb{B}-\mathbb{A}-T)S^{e}(T-\mathbb{A})=
\sum_{h=0}^{e}S^{e-h}(-\mathbb{A})S_{(e-d+k)^{k-1}}(\mathbb{B}-\mathbb{A}-T)T^h.
$$
Again by Lemma \ref{lemma: Lascoux 148}, we have
$$
S_{(e-d+k)^{k-1}}(\mathbb{B}-\mathbb{A}-T)T^h =S_{(e-d+k)^{k-1};\,
h}(\mathbb{B}-\mathbb{A}; T),
$$
where
\begin{equation}\label{eq: prop: main: determinat_2}
S_{(e-d+k)^{k-1};\, h}(\mathbb{B}-\mathbb{A}; T)=\det \left(
 \begin{array}{cccc}
 S^{e-d+k}(\mathbb{A-\mathbb{B}})  & \cdots & S^{e-d+2k-2}(\mathbb{A-\mathbb{B}}) & T^{h+k-1} \\
 S^{e-d+k-1}(\mathbb{A-\mathbb{B}})  & \ddots & \vdots & \vdots \\
 \vdots & \ddots & S^{e-d+k}(\mathbb{A-\mathbb{B}})  & T^{h+1} \\
 S^{e-d+1}(\mathbb{A-\mathbb{B}}) & \cdots & S^{e-d+k-1}(\mathbb{A-\mathbb{B}})  & T^h \\
 \end{array}
 \right).
\end{equation}

Denote by $A_{j, \, d-k}$ the coefficient of the monomial $T^{d-k}$
in $S_{k^{e-d+k-1};\, j}(\mathbb{A}-\mathbb{B}; T)$ for $0\le j \le
d$, considering $S_{k^{e-d+k-1};\, j}(\mathbb{A}-\mathbb{B}; T)$ as
an element of $\fq[\mathbb{A},\mathbb{B}][T]$. Further, denote by
$B_{h, \, d-k}$ the coefficient of the monomial $T^{d-k}$ in
$S_{(e-d+k)^{k-1};\, h}(\mathbb{B}-\mathbb{A}; T)$ for $0\le h \le
e$.  By \eqref{eq: schur_remainder_formula} we have
$$
F_k=(-1)^{d-k+1}\sum_{h=0}^{e}S^{e-h}(-\mathbb{A})B_{h, \,d-k} +
\sum_{j=0}^{d}S^{d-j}(-\mathbb{B})A_{j,\, d-k}.
$$
By \eqref{eq: prop: main: determinat_1} and \eqref{eq: prop: main:
determinat_2} it is clear that $A_{j,\, d-k}=0$ for $j> d-k$ and $j+e-d+k-1<d-k$ and
$B_{h, \,d-k}=0$ for $h> d-k$ and $h+k-1<d-k$. Thus,
\begin{equation}\label{eq: F_k_formula}
F_k=(-1)^{d-k+1}\sum_{h=\max\{d-2k+1,
\,0\}}^{d-k}S^{e-h}(-\mathbb{A})B_{h, \,d-k} + \sum_{j=\max
\{2d-e-2k+1,\,0\}}^{d-k}S^{d-j}(-\mathbb{B})A_{j,\, d-k}.
\end{equation}
Now, $A_{j,\, d-k}$ and $B_{h, \,d-k}$ are the determinants of the
submatrices obtained by removing the last column and the row
corresponding to $T^{d-k}$ in the matrices of \eqref{eq: prop: main:
determinat_1} and \eqref{eq: prop: main: determinat_2} respectively.
More precisely,
%
$$
A_{j,\, d-k}=\det \left(
 \begin{array}{cccccc}  S^{k}(\mathbb{A}-\mathbb{B}) & \cdots \\
  \vdots    &  \ddots \\
      & & S^{k}(\mathbb{A}-\mathbb{B})& \cdots\\
      & & \vdots& S^{k-1}(\mathbb{A-\mathbb{B}})& \cdots \\
      & &  &\vdots & \ddots \\
      & & & & &  S^{k-1}(\mathbb{A-\mathbb{B}})
 \end{array}\right),
$$
with the first $j+e-2d+2k-1$ columns having
$S^{k}(\mathbb{A}-\mathbb{B})$ in the diagonal and the last $d-k-j$
columns having $S^{k-1}(\mathbb{A-\mathbb{B}})$ in the diagonal, and
$$
B_{h, \,d-k}=\det \left(\begin{array}{cccccc}
    S^{e-d+k}(\mathbb{A}-\mathbb{B})\!\!\!\!&\cdots \\
\vdots    & \!\!\!\!\!\!\!\!\ddots \\
    & &\!\!\!\!\!\!\!\!S^{e-d+k}(\mathbb{A}-\mathbb{B})&\cdots\\
    & &\vdots &\!\!\!\!S^{e-d+k-1}(\mathbb{A-\mathbb{B}})&\cdots \\
    & & & \vdots& \!\!\!\!\!\!\!\!\ddots \\
    & & & & &  \!\!\!\!\!\!\!\!S^{e-d+k-1}(\mathbb{A-\mathbb{B}})
   \end{array}\right),
$$
with the first $h-d+2k-1$ columns having
$S^{e-d+k}(\mathbb{A}-\mathbb{B})$ in the diagonal and the last
$d-k-h$ columns having $S^{e-d+k-1}(\mathbb{A-\mathbb{B}})$ in the
diagonal.

Considering $A_{j,\, d-k}$ and $B_{h, \,d-k}$ as polynomials in the
variables $\mathbf S :=(S^{i}(-\mathbb{B}): \, 1\le i \le d)$ and
coefficients in $\fq[\mathbf T]$, by their determinantal expressions
we conclude that
\begin{equation}\label{eq: S_degrees}
\deg_{\mathbf S}A_{j,\, d-k}\le e-d+k-1, \quad \deg_{\mathbf S}B_{h,
\,d-k}\le k-1.
\end{equation}
Combining these upper bounds and \eqref{eq: F_k_formula} we readily
see that
\begin{equation}\label{eq: total deg_F_k_inequality}
\deg_{\mathbf S}F_k\le e-d+k.
\end{equation}
\begin{claim} $A_{d-k, \, d-k}$ is a nonzero element of
$\fq[\mathbf T][\mathbf S]$ with $\deg_{\mathbf S}A_{d-k, \,
d-k}=e-d+k-1$. Further, $A_{d-k, \, d-k}$ is monic of degree
$e-d+k-1$ in $S^{k}(-\mathbb{B})$.\end{claim}
\begin{proof}[Proof of Claim.] Let $N:=e-d+k-1$. Observe
that the determinantal expression of $A_{d-k, \, d-k}$ consists of
$e-d+k-1$ columns having $S^{k}(\mathbb{A}-\mathbb{B})$ in the
diagonal. More precisely, $A_{d-k, \, d-k}=\det (a_{ij})_{1 \le i,j
\le N}$, where $ a_{ij}:=S^{k+j-i}(\mathbb{A}-\mathbb{B}) $. We
remark that
$$
S^{h}(\mathbb{A}-\mathbb{B}) = S^{h}(-\mathbb{B})+
S^{1}(\mathbb{A})S^{h-1}(-\mathbb{B})+ \cdots + S^{h}(\mathbb{A})
$$
is a polynomial of degree one in the variables $\mathbf S$ for $1\le
h \le N$. Further,  $S^{k}(\mathbb{A}-\mathbb{B})$ is monic of
degree one in $S^{k}(-\mathbb{B})$ and
$S^{h}(\mathbb{A}-\mathbb{B})$ is of degree zero in
$S^{k}(-\mathbb{B})$ for $h< k$. Write
$$
A_{d-k, \, d-k}=\sum_{\sigma}\pm a_{1\sigma_1}a_{2\sigma_2}\cdots
a_{N\sigma_N},
$$
where $\sigma$ runs over all permutations of $(1,2,\dots, N)$. By
the previous remarks we see that
$$\deg_{\mathbf S}A_{d-k, \, d-k}\le N.$$
To prove the equality, consider a permutation $(\sigma_1, \dots,
\sigma_N) \neq (1,2,\dots, N)$. Then there exists an index $i$ with
$\sigma_i < i$. For such an index, since $k+\sigma_i-i < k$, the
entry $a_{i\sigma_i}=S^{k+\sigma_i-i}(\mathbb{A}-\mathbb{B})$ has
degree zero in $S^{k}(-\mathbb{B})$. Thus
$a_{1\sigma_1}a_{2\sigma_2}\cdots a_{N\sigma_N}$ has degree at most
$N-1$ in $S^{k}(-\mathbb{B})$. On the other hand, the term $
a_{11}\cdots a_{NN}=S^{k}(\mathbb{A}-\mathbb{B})^{N} $ is monic of
degree $N$ in $S^{k}(-\mathbb{B})$. This implies the claim.
\end{proof}

Write the second sum in \eqref{eq: F_k_formula} as
\begin{align*}
\sum_{j=\max\{2d-e-2k+1,\,0\}}^{d-k}\!\!\!\!\!\!S^{d-j}(-\mathbb{B})A_{j,\,
d-k}=S^{k}(-\mathbb{B})A_{d-k, \, d-k} +\!\!\!\!
\sum_{j=\max\{2d-e-2k+1,
\,0\}}^{d-k-1}\!\!\!\!\!\!S^{d-j}(-\mathbb{B})A_{j,\,d-k}.
\end{align*}
According to the claim, the polynomial  $S^{k}(-\mathbb{B})A_{d-k,
\, d-k}$ is monic of degree $e-d+k$ in $S^{k}(-\mathbb{B})$.
Further, by \eqref{eq: S_degrees}, taking into account that
$S^{d-j}(-\mathbb{B})\neq S^k(-\mathbb{B})$ for $0\le j \le d-k-1$
it follows that
$$
\deg_{S^k(-\mathbb{B})}\Bigg( \sum_{j=\max\{2d-e-2k+1,
\,0\}}^{d-k-1}S^{d-j}(-\mathbb{B})A_{j,\,d-k}\Bigg)\le e-d+k-1.
$$
Therefore, $\sum_{j=\max\{2d-e-2k+1,
\,0\}}^{d-k}S^{d-j}(-\mathbb{B})A_{j,\, d-k}$ is monic of degree
$e-d+k$ in $S^{k}(-\mathbb{B})$. On the other hand, according to
\eqref{eq: S_degrees}, the first sum in the right-hand side of
\eqref{eq: F_k_formula} has degree at most $k-1$, and then less than
$e-d+k-1$, in $S^k(-\mathbb{B})$. We conclude that $F_k$ is monic of
degree $e-d+k$ in $S^{k}(-\mathbb{B})$. This together with
\eqref{eq: total deg_F_k_inequality} implies that $\deg_{\mathbf S}
F_k=e-d+k$, which finishes the proof of the proposition.
\end{proof}

%
 %

\subsection{Specialization of the generic case}
\label{sec: specialization}
As expressed in Section \ref{subsec: remainder as Schur functions},
for {\em generic} input elements
$S^{e}(T-\mathbb{A})\in\fq[\mathbb{A}][T]$ and
$S^{d}(T-\mathbb{B})\in\fq[\mathbb{B}][T]$, the Euclidean algorithm
performs $d$ steps. Further, if the Euclidean algorithm is applied
to polynomials $g,f\in\fq[T]$ with $\deg g=e$ and $\deg f=d$, and
performs $d$ steps, then the degrees of the successive remainders
decrease by 1 each step, and the sequences of quotients and
remainders associated to $g$ and $f$ coincide with the
specialization of the sequences associated to $S^{e}(T-\mathbb{A})$
and $S^{d}(T-\mathbb{B})$. The next result shows that, if the
sequence of remainders associated to polynomials $g,f\in\fq[T]$ with
$\deg g=e$ and $\deg f=d$ fails to have the degree pattern of the
generic case, the first remainder where such a failure occurs is
still a specialization of the corresponding one of the generic case.
\begin{lemma} \label{lemma: specialization}
For a specialization $\mathbb{A}\mapsto \overline{a}$ and
$\mathbb{B}\mapsto \overline{b}$ in $\fq$, denote by $r_1, \dots,
r_{k}$ the first $k$ remainders of the application of the Euclidean
algorithm to $S^{e}(T-\overline{a})$ and $S^{d}(T-\overline{b})$. If
$\deg r_i=d-i$ for $1\le i \le k-1$, then
$r_i=\mathcal{R}_i(\overline{a}, \overline{b})$ for $1\le i \le k$.
Further, if $\deg \mathcal{R}_i(\overline{a}, \overline{b})=d-i$ for
$1\le i \le k-1$, then $r_i=\mathcal{R}_i(\overline{a},
\overline{b})$ for $1\le i \le k$.
\end{lemma}
\begin{proof}
Substituting $\overline{a}$ for $\mathbb{A}$ and $\overline{b}$ for
$\mathbb{B}$ in the first identity of \eqref{eq: euclidean_alg} we
easily see that $\mathcal{R}_1(\overline{a}, \overline{b})$ is the
remainder in the division of $S^{e}(T-\overline{a})$ by
$S^{d}(T-\overline{b})$, which proves that
$r_1=\mathcal{R}_1(\overline{a}, \overline{b})$. Let $j>1$ and
assume inductively that $r_i=\mathcal{R}_i(\overline{a},
\overline{b})$ for $1\le i \le j <k$. Thus $\deg
\mathcal{R}_i(\overline{a}, \overline{b})=d-i$ for $1\le i \le j$.
Taking into account that $F_i$ is the leading coefficient of
$\mathcal{R}_i$ for $1 \le i \le k$, we deduce that
$F_i(\overline{a}, \overline{b})\neq 0$ for $1\le i \le j$. Since
$q_{j+1}\in \fq[\mathbb{A},\mathbb{B}]_{F_j}[T]$, where
$\fq[\mathbb{A},\mathbb{B}]_{F_j}$ is the localization of
$\fq[\mathbb{A},\mathbb{B}]$ at $F_j$, we can substitute
$\overline{a}$ for $\mathbb{A}$ and $\overline{b}$ for $\mathbb{B}$
in the $(j+1)$th equation of \eqref{eq: euclidean_alg} to obtain
$$
r_{j-1}=q_{j+1}(\overline{a},
\overline{b})r_j+\mathcal{R}_{j+1}(\overline{a}, \overline{b}).
$$
Since
$$
\deg_T\mathcal{R}_{j+1}(\overline{a}, \overline{b})\le
\deg_T\mathcal{R}_{j+1}=d-j-1< \deg_T\mathcal{R}_j(\overline{a},
\overline{b}),
$$
we conclude that $\mathcal{R}_{j+1}(\overline{a}, \overline{b})$ is
the remainder in the division of $r_{j-1}$ by $r_j$. In other words,
$r_{j+1}=\mathcal{R}_{j+1}(\overline{a}, \overline{b})$, which
completes the proof of the first assertion of the lemma. The second
assertion is proved with a similar argument.
\end{proof}


Let $\overline{a}\in \cfq{\!}^{e}$ be the tuple of roots of $g$ in
any order, so that $g=S^{e}(T-\overline{a})$. Let
$G_k:=F_k(\mathbf{S}(\overline{a}),\mathbf{S}(-\mathbb{B}))$ denote
the polynomial obtained by substituting $\overline{a}$ for
$\mathbb{A}$ in $F_k$. Since the set $\mathbf T:=(S^i(\mathbb{A}):
1\le i \le e)$ consists of the first $e$ complete symmetric
functions in $\mathbb{A}$, it follows that $\mathbf T(\overline{a})$
belongs to $\fq^{e}$, and thus $G_k$ belongs to $\fq[\mathbf S]$.
Further, Proposition \ref{prop: main} shows that $G_k$ is a nonzero
polynomial with $\deg_{\mathbf S}G_k=e-d+k$, which is monic in
$S^k(-\mathbb{B})$ with $\deg_{S^k(-\mathbb{B})}G_k= e-d+k$.

We end this section with a result which will be crucial to establish
lower bounds for the average-case complexity of the Euclidean
algorithm. 
As we shall see in the next section, for a fixed $g\in\fq[T]$ with
$\deg g=e$, a random element $f\in \fq[T]$ with $\deg f=d$ and $g$
are relatively prime with high probability. In this sense, we call a
polynomial $f\in\fq[T]$ with $\deg f=d$ {\em generic} (with respect
to $g$) if the remainder sequence in the Euclidean algorithm applied
to the pair $(g,f)$ has length $d$. In particular, in such a
remainder sequence $(r_1,\ldots,r_d)$ we have $\deg(r_k)=d-k$ for
$1\le k\le d$. The next result establishes a lower bound on the
number generic monic elements in $\fq[T]$ of degree $d$.
\begin{proposition}\label{prop: lower bound generic pols}
Let $\mathcal{G}\subset\fq[T]$ be the set of monic elements of
degree $d$ which are generic in the sense above. Then
$$|\mathcal{G}|\ge q^d\Bigg(1-\frac{d(2e-d+1)}{2q}\Bigg).$$
In particular, for $q>d(2e-d+1)/2$ the set $\mathcal{G}$ is
nonempty.
\end{proposition}
\begin{proof}
Let $f:=T^d+s_1T^{d-1}+\cdots+s_d\in \mathcal{G}$ and let
$(r_1,\ldots,r_d)$ be the sequence of remainders in the Euclidean
algorithm applied to the pair $(g,f)$. By hypothesis $\deg(r_j)=d-j$
for $1\le j\le d$, which by Lemma \ref{lemma: specialization} is
equivalent to the condition $G_j(s_1, \dots, s_d)\not=0$ for $1\le
j\le d$. It follows that
$$\mathcal{G}=\bigcap_{j=1}^d\big(\fq^d\setminus\mathcal{V}(G_j)(\fq)\big)
=\fq^d\setminus\bigcup_{j=1}^d\mathcal{V}(G_j)(\fq).$$
As a consequence,
$$|\mathcal{G}|=q^d-\Bigg|\bigcup_{j=1}^d\mathcal{V}(G_j)(\fq)\Bigg|
\ge q^d-\sum_{j=1}^d|\mathcal{V}(G_j)(\fq)|.$$
According to \eqref{eq: upper bound -- affine gral}, we have
$$
\sum_{k=1}^d|\mathcal{V}(G_k)(\fq)|\leq {q^{d-1}}\sum_{k=1}^d
(e-d+k)= q^{d-1}\frac{d(2e-d+1)}{2},$$
which readily implies the proposition.
\end{proof}
%
%
\section{Analysis of the average degree in the Euclidean algorithm}
\label{sec: average degree}
Let $e, d$ be positive integers with $e>d$. For any $m\ge 0$, we
denote by $\fq[T]_m$ the set of monic polynomials of degree $m$ with
coefficients in $\fq$. For a fixed $g\in \fq[T]_e$, let
$\mathcal{X}_g: \fq[T]_d \rightarrow \{0,\ldots,d\}$,
$\mathcal{X}_g(f)=\deg (\gcd(g,f))$ be the random variable defined
by the degree of the greatest common divisor $\gcd(g,f)$, where
$\fq[T]_d$ is endowed with the uniform probability. Applying the
Euclidean algorithm to a pair $(g,f)$ with $f\in\fq[T]_d$ we obtain
a positive integer $k$ with $1\leq k\leq d$, a unique polynomial
quotient sequence $(q_1,\ldots,q_{k+1})$ and a unique polynomial
remainder sequence $(r_1,\ldots,r_k)$, satisfying the following
conditions:
\begin{align*}
g&=  f\cdot q_1+r_1, & \deg(r_1)<&\deg(f),\\
f&=r_1\cdot q_2+r_2, &
\deg(r_2)<&\deg(r_1),\\
&\ \ \vdots&\quad \vdots \\
r_{k-2}&=r_{k-1}\cdot q_k+r_k,
&\deg(r_k)<&\deg(r_{k-1}),\\
r_{k-1}&=r_k\cdot q_{k+1}.
\end{align*}
First we study the average degree of the gcd, namely the expected
value of $\mathcal{X}_g$:
\begin{equation}\label{eq: def expected value degree}
E[\mathcal{X}_g]=\sum_{i=0}^di\,\frac{|B_i|}{q^d}=\sum_{i=1}^di\,\frac{|B_i|}{q^d}
=\sum_{i=1}^d\sum_{j=i}^d\frac{|B_j|}{q^d},
\end{equation}
where $B_i:=\{f\in \fq[T]_d:\,\mathcal{X}_g(f)=i\}$ for $0\le i\le
d$.

%
%

We start with an estimate on
$\sum_{j=1}^d{|B_j|}=\big|\bigcup_{j=1}^dB_j\big|$. For this
purpose, observe that
$$
\bigcup_{j=1}^dB_j=\{f \in \fq[T]_d:\mathrm{res}(g,f)=0\},
$$
where $\mathrm{res}(\cdot,\cdot)$ denotes resultant.
We recall that $g$ has factorization pattern
$(\lambda_1,\ldots,\lambda_e)\in\Z_{\ge 0}^e$, with
$\lambda_1+2\,\lambda_2+\cdots+e\,\lambda_e=e$, if $g$ has
$\lambda_i$ irreducible factors in $\fq[T]$ of degree $i$ (counting
multiplicities) for $1\le i\le e$. We shall also consider the {\em
reduced} factorization pattern
$(\lambda_1^*,\ldots,\lambda_e^*)\in\Z_{\ge 0}^e$ of $g$, where
$\lambda_i^*$ is the number of distinct irreducible factors of $g$
in $\fq[T]_i$ for $1\le i\le e$, and denote by $g^*$ the square-free
part of $g$, namely the product of all distinct irreducible factors
of $g$ (without multiplicities). In particular, we have that
$(\lambda_1^*,\ldots,\lambda_e^*)$ is the factorization pattern of
$g^*$. We have the following result.
\begin{proposition}\label{prop: estimate B_0}
Let $e,d$ be integers with $e>d>0$. Let $g$ be an element of
$\fq[T]_e$, $g^*$ its square-free part and
$(\lambda_1^*,\ldots,\lambda_e^*)$ the factorization pattern of
$g^*$. Let $k$ be the least integer with $\lambda_k^*>0$. If $k\le
d$, then
$$
\lambda_k^*\,q^{d-k}-\binom{\lambda_k^*}{2}q^{\max\{d-2k,0\}} \le
\Bigg|\bigcup_{j=1}^dB_j\Bigg|\le
\lambda_k^*\,q^{d-k}+\sum_{i=k+1}^d\lambda_i^*\,q^{d-i}.
$$
\end{proposition}
\begin{proof}
For $f\in\fq[T]_d$ we have $\mathrm{res}(g,f)=0$ if and only if
$\mathrm{res}(g^*,f)=0$. As a consequence, we shall consider the
resultant $\mathrm{res}(g^*,f)$. Denote by
$g_i:=\prod_{j=1}^{\lambda_i^*}g_{i,j}$ the product of all
irreducible factors of $g^*$ of degree $i$ for $1\le i\le e$. Let
$\bfs S:=(S_{d-1},\ldots,S_0)$ be a vector of indeterminates and
$$F(\bfs S,T):=T^d+S_{d-1}T^{d-1}+\cdots+S_0.$$
The product formula for the resultant (see, e.g., \cite[Theorem
4.16]{BaPoRo06}) implies
$$\mathrm{res}(g^*,F(\bfs S,T))=\prod_{i=k}^e{\sf R}_i:=
\prod_{i=k}^e\mathrm{res}(g_i,F(\bfs S,T))
=\prod_{i=k}^e\prod_{j=1}^{\lambda_i^*}\mathrm{res}(g_{i,j},F(\bfs
S,T)).$$

Now, for any $i$ with $k\le i\le d$ and $\lambda_i^*>0$, we have
$${\sf R}_i:=\prod_{j=1}^{\lambda_i^*}{\sf R}_{i,j},\quad
{\sf R}_{i,j}:=\mathrm{res}(g_{i,j},F(\bfs
S,T)).$$
Since $g_{i,j}$ is an irreducible element of $\fq[T]$, for $\bfs
s\in\fq^d$ we have ${\sf R}_{i,j}(\bfs s)=0$ if and only if
$g_{i,j}$ divides $F(\bfs s,T)$. Further, as $\{F(\bfs s,T):\bfs
s\in\fq^d\}\subset\fq[T]_d$, we conclude that there is a bijection
between the set of $\fq$-rational zeros of ${\sf R}_{i,j}$ and the
set of multiples in $\fq[T]_d$ of $g_{i,j}$. As the latter has
cardinality $q^{d-i}$, we conclude that
$|\mathcal{V}({\sf R}_{i,j})(\fq)|= q^{d-i}$.
Therefore,
$$
|\mathcal{V}({\sf
R}_i)(\fq)|=\Bigg|\bigcup_{j=1}^{\lambda_i^*}\mathcal{V}({\sf
R}_{i,j})(\fq)\Bigg|\le \lambda_i^*\,q^{d-i}.
$$

On the other hand, for $i>d$ with $\lambda_i^*>0$, there is no
element of $\fq[T]_d$ having a nontrivial common factor with $g_i$
defined over $\fq$. This implies that the set $\mathcal{V}({\sf
R}_i)(\fq)$ is empty, namely
$$
|\mathcal{V}({\sf R}_i)(\fq)|=0.
$$

Now we focuss on the case $i=k$. If ${\sf
R}_k:=\prod_{j=1}^{\lambda_k^*}{\sf R}_{k,j}$, we have
$$|\mathcal{V}({\sf R}_k)(\fq)|\le |\mathcal{V}(\mathrm{res}(g,F(\bfs S,T)))(\fq)|\le
|\mathcal{V}({\sf R}_k)(\fq)|+\sum_{i=k+1}^d\lambda_i^*\,q^{d-i}.$$
Our previous argument shows that $\mathcal{V}({\sf R}_k)(\fq)$ is a
union of $\lambda_k^*$ sets $\mathcal{V}({\sf R}_{k,j})(\fq)$ of
cardinality $q^{d-k}$, which are pairwise distinct.
Further, $\bfs s\in \mathcal{V}({\sf R}_{k,j_1})(\fq)\cap
\mathcal{V}({\sf R}_{k,j_2})(\fq)$ for $j_1\not= j_2$ if and only if
both $g_{k,j_1}$ and $g_{k,j_2}$ divide  $F(\bfs s,T)$. As
$g_{k,j_1}$ and $g_{k,j_2}$ are two distinct irreducible elements of
$\fq[T]$, this holds if and only if $g_{k,j_1}\cdot g_{k,j_2}$
divides $F(\bfs s,T)$. It follows that
$$|\mathcal{V}({\sf R}_{k,j_1})(\fq)\cap \mathcal{V}({\sf
R}_{k,j_2})(\fq)|=\left\{\begin{array}{cc} q^{d-2k}&\text{for }d\ge
2k,\\0&\text{for }d<2k.\end{array}\right.$$
In particular, the Bonferroni inequalities imply
$$\lambda_k^*\,q^{d-k}-\binom{\lambda_k^*}{2}q^{\max\{d-2k,0\}}\le |\mathcal{V}({\sf
R}_k)(\fq)|=\Bigg|\bigcup_{j=1}^{\lambda_k^*}\mathcal{V}({\sf
R}_{k,j})(\fq)\Bigg|\le \lambda_k^*\,q^{d-k}.$$
From this the statement of the proposition readily follows.
\end{proof}

%
As an immediate consequence of Proposition \ref{prop: estimate B_0}
we obtain an estimate on the probability that a random element of
$\fq[T]_d$ is relatively prime with $g$.
\begin{theorem}\label{teo: prob coprime} Let $e,d$ be integers with $e>d>0$.
Let $g$ be an element of $\fq[T]_e$, $g^*$ its square-free part and
$(\lambda_1^*,\ldots,\lambda_e^*)$ the factorization pattern of
$g^*$. Let $k$ be the least integer with $\lambda_k^*>0$. If $k\le
d$, then the probability $\mathcal{P}_0:={|B_0|}/{q^{d}}$ that a
random element $f\in \fq[T]_d$ and $g$ are relatively prime is
bounded in the following way:
$$
1 - \frac{\lambda_k^*}{q^k} -\sum_{i=k+1}^d\frac{\lambda_i^*}{q^i}
\le \mathcal{P}_0\le 1 -
\frac{\lambda_k^*}{q^k}+\binom{\lambda_k^*}{2}\frac{1}{q^{\min\{2k,d\}}}.
$$
In particular, for $q>2e$ we have $\mathcal{P}_0
>\frac{1}{2}$.
\end{theorem}
\begin{proof}
Observe that
$$
|B_0|=\big|\{f \in \fq[T]_d:
\gcd(g,f)=1\}\big|=\Bigg|\fq^d\setminus\bigcup_{j=1}^dB_j\Bigg|=q^d
-\Bigg|\bigcup_{j=1}^dB_j\Bigg|.
$$
Then the statement readily follows from Proposition \ref{prop:
estimate B_0}.
\end{proof}

If the square-free part $g^*$ of $g\in\fq[T]_e$ has a factorization
pattern $(\lambda_1^*,\ldots,\lambda_e^*)$ as in Theorem \ref{teo:
prob coprime}, then all its irreducible factors have degree at least
$k$. It follows that $B_1\cup\cdots\cup B_{k-1}$ is the empty set,
which implies the following corollary.
\begin{corollary}
\label{coro: estimate cardinal B_j cup B_d for j le k}
With hypotheses as in Theorem \ref{teo: prob coprime}, for $1\le
i\le k$, we have
$$
\lambda_k^*\,q^{d-k}-\binom{\lambda_k^*}{2}q^{\max\{d-2k,0\}} \le
\Bigg|\bigcup_{j=i}^dB_j\Bigg|\le
\lambda_k^*\,q^{d-k}+\sum_{j=k+1}^d\lambda_j^*\,q^{d-j}.
$$
\end{corollary}

Next we bound the sum of the cardinalities of $\bigcup_{j=i}^d B_j$
for $i\ge k+1$.
\begin{proposition}\label{prop: upper bound deg gcd ge j}
Let $g\in\fq[T]_e$ have a factorization pattern
$(\lambda_1,\ldots,\lambda_e)$ and let $k$ be the least index with
$\lambda_k>0$. We have
$$
\sum_{i=k+1}^d\Bigg|\bigcup_{j=i}^dB_j\Bigg|\le
\sum_{i=k+1}^d(i-k)\,q^{d-i}\sum_{\stackrel{\scriptstyle
h_k\le\lambda_k,\ldots,\,h_i\le\lambda_i}{k\,h_k+\cdots+i\,h_i=i}}
\binom{\lambda_k}{h_k}\cdots \binom{\lambda_i}{h_i}.
$$
%
%
\end{proposition}
\begin{proof}
Observe that
$$B_i\cup\cdots\cup B_d=\{f\in\fq[T]_d:\deg\gcd(g,f)\ge i\}.$$
%
%
%
Fix a factor $m\in\fq[T]_j$
of degree $j\ge i$ of $g$. 
Then the set $L_m$ of multiples $f\in\fq[T]_d$ of $m$ 
has cardinality $|L_m|=q^{d-j}$. As a consequence, letting $m$ vary
over the set of factors in $\fq[T]_j$ of $g$ we conclude that
$$\Bigg|\bigcup_{j=i}^dB_j\Bigg|\le \sum_{j=i}^d\eta_j\,q^{d-j},$$
where $\eta_j$ is the number of distinct factors of $g$ in
$\fq[T]_j$ for $i\le j\le d$.
%
%
%
%
It follows that
$$\sum_{i=k+1}^d\Bigg|\bigcup_{j=i}^dB_j\Bigg|\le \sum_{i=k+1}^d(i-k)
\,\eta_i\,q^{d-i}.$$
It remains to express the $\eta_i$ in terms of
$\lambda_1,\ldots,\lambda_d$. For this purpose, we observe that
$$\eta_i\le[X^i]\Bigg(\prod_{j=k}^i(1+X^j)^{\lambda_j}\Bigg)
=\sum_{\stackrel{\scriptstyle
h_k\le\lambda_k,\ldots,\,h_i\le\lambda_i}{k\,h_k+\cdots+i\,h_i=i}}
\binom{\lambda_k}{h_k}\cdots \binom{\lambda_i}{h_i},$$
where $[X^i]f$ denotes the coefficient of $X^i$ in the monomial
expansion of $f\in\K[X]$. This proves the proposition.
\end{proof}


Now we obtain an estimate for the average degree of $\gcd(g,f)$ for
random $f\in\fq[T]_d$.
\begin{theorem}\label{teo: average degree mcd}
Let $e, d$ be integers with $e>d>0$, $g$ an element of $\fq[T]_e$
with factorization pattern $(\lambda_1,\ldots,\lambda_e)$ and $k$
the least index with $\lambda_k>0$. Denote by $\lambda_k^*$ the
number of distinct irreducible factors of $g$ in $\fq[T]_k$. If
$k\le d$, then the average degree $E[\mathcal{X}_g]$ of the greatest
common divisor of $g$ and a random element $f$ of $\fq[T]_d$ is
bounded in the following way:
$$\frac{k\,\lambda_k^*}{q^k}-\binom{\lambda_k^*}{2}\frac{k}{q^{\min\{2k,d\}}}\le
E[\mathcal{X}_g]\le
\frac{k\,\lambda_k^*}{q^k}+\sum_{i=k+1}^d\frac{i}{q^i}
\sum_{\stackrel{\scriptstyle
h_k\le\lambda_k,\ldots,\,h_i\le\lambda_i}{k\,h_k+\cdots+i\,h_i=i}}
\binom{\lambda_k}{h_k}\cdots \binom{\lambda_i}{h_i}.$$
\end{theorem}
\begin{proof}
According to \eqref{eq: def expected value degree},
$$E[\mathcal{X}_g]=\sum_{i=1}^d\sum_{j=i}^d\frac{|B_j|}{q^d}.$$

By Proposition \ref{prop: estimate B_0} and Corollary \ref{coro:
estimate cardinal B_j cup B_d for j le k}, for $1\le i\le k$,
$$
\lambda_k^*\,q^{d-k}-\binom{\lambda_k^*}{2}q^{\max\{d-2k,0\}} \le
\Bigg|\bigcup_{j=i}^dB_j\Bigg|\le
\lambda_k^*\,q^{d-k}+\sum_{i=k+1}^d\lambda_i\,q^{d-i}.$$
By Proposition \ref{prop: upper bound deg gcd ge j}, we have
$$
\sum_{i=k+1}^d\Bigg|\bigcup_{j=i}^dB_j\Bigg|\le
\sum_{i=k+1}^d(i-k)\,q^{d-i} \sum_{\stackrel{\scriptstyle
h_k\le\lambda_k,\ldots,\,h_i\le\lambda_i}{k\,h_k+\cdots+i\,h_i=i}}
\binom{\lambda_k}{h_k}\cdots \binom{\lambda_i}{h_i}.
$$
We conclude that
\begin{align*}
\frac{k\,\lambda_k^*}{q^k}-\binom{\lambda_k^*}{2}\frac{k}{q^{\min\{2k,d\}}}\le
E[\mathcal{X}_g]\le
&\frac{k\,\lambda_k^*}{q^k}+\sum_{i=k+1}^d\frac{k\,\lambda_i}{q^i}
\\&\ + \sum_{i=k+1}^d\frac{i-k}{q^i} \sum_{\stackrel{\scriptstyle
h_k\le\lambda_k,\ldots,\,h_i\le\lambda_i}{k\,h_k+\cdots+i\,h_i=i}}
\binom{\lambda_k}{h_k}\cdots \binom{\lambda_i}{h_i}\\
\le& \frac{k\,\lambda_k^*}{q^k}+ \sum_{i=k+1}^d\frac{i}{q^i}
\sum_{\stackrel{\scriptstyle
h_k\le\lambda_k,\ldots,\,h_i\le\lambda_i}{k\,h_k+\cdots+i\,h_i=i}}
\binom{\lambda_k}{h_k}\cdots \binom{\lambda_i}{h_i},
\end{align*}
which proves the theorem.
\end{proof}

To simplify the upper bound of Theorem \ref{teo: average degree mcd}
we recall that the inner sum in such an upper bound is actually an
upper bound for the number $\eta_i$ of distinct factors of $g$ in
$\fq[T]_i$, namely
$$\eta_i\le[X^i]\Bigg(\prod_{j=k}^i(1+X^j)^{\lambda_j}\Bigg)
=\sum_{\stackrel{\scriptstyle
h_k\le\lambda_k,\ldots,\,h_i\le\lambda_i}{k\,h_k+\cdots+i\,h_i=i}}
\binom{\lambda_k}{h_k}\cdots \binom{\lambda_i}{h_i},$$
with equality when $g$ is square-free.
Using the generalized Vandermonde identity (see, e.g., \cite[page
248]{GrKnPa94}), we have
\begin{equation}\label{eq: eta_i - first bound}
\eta_i\le \sum_{\stackrel{\scriptstyle
h_k\le\lambda_k,\ldots,\,h_i\le\lambda_i}{k\,h_k+\cdots+i\,h_i=i}}
\binom{k\,\lambda_k}{k\,h_k}\cdots \binom{i\,\lambda_i}{i\,h_i}\le
\binom{k\,\lambda_k+\cdots+i\,\lambda_i}{i}.
\end{equation}
On the other hand, taking into account that the expansion of the
analytic function $h:\C\to\C$,
$h(z):=\prod_{j=k}^i(1+z^j)^{\lambda_j}$ has non-negative
coefficients at $0$, from, e.g., \cite[Proposition IV.1]{FlSe09} we
conclude that
\begin{equation}\label{eq: eta_i - second bound}
\eta_i\le h(1)=2^{\lambda_k+\cdots+\lambda_i}.
\end{equation}
The accuracy of \eqref{eq: eta_i - first bound} and \eqref{eq: eta_i
- second bound} depends on the actual factorization pattern
$(\lambda_1,\ldots,\lambda_e)$. For example, if $g\in\fq[T]_e$ is a
polynomial with an ``equal-degree factorization'' (that is,
$k\lambda_k=e$), then for large $k$ the bound \eqref{eq: eta_i -
second bound} is preferable, while for large $\lambda_k$ the bound
\eqref{eq: eta_i - first bound} is more accurate.

Finally, for the results on the average-case complexity of the
Euclidean algorithm we shall use a further upper bound on
$E[\mathcal{X}_g]$. This bound, although not as precise as the one
of Theorem \ref{teo: average degree mcd}, has a simple expression
which suffices for the purposes of the next section.
\begin{lemma}\label{lemma: average degree}
Let $e, d$ be integers with $e>d>0$ and let $g\in\fq[T]_e$. Then
$$E[\mathcal{X}_g]\le \frac{de}{q^k}.$$
\end{lemma}
\begin{proof}
Let $(\lambda_1,\ldots,\lambda_e)$ be the factorization pattern of
$g$ and let $k$ be the least index with $\lambda_k>0$. By
Proposition \ref{prop: estimate B_0},

$$E[\mathcal{X}_g]=\sum_{k=1}^d\sum_{j=k}^d\frac{|B_j|}{q^d}\le
\sum_{k=1}^d\sum_{j=1}^d\frac{|B_j|}{q^d} \le d
\Bigg(\frac{\lambda_k}{q^k}+\sum_{i=k+1}^d\frac{\lambda_i}{q^i}\Bigg)
\le \frac{de}{q^k}.$$
\end{proof}
%
%
\section{Average-case analysis of the Euclidean Algorithm}
\label{sec: average case analysis}
Let $e, d$ be positive integers with $q>d(2e-d+1)/2$ and $e>d$ and
let $g\in \fq[T]_e$ be fixed. In this section we analyze the
average-case complexity of the Euclidean algorithm applied to pairs
$(g,f)$ with $f\in \fq[T]_d$.

Given positive integers $m, n$ with $m > n$ and $(f_1,f_2) \in
\fq[T]_m \times(\fq[T]_n\setminus\{0\})$, and an arithmetic
operation ${\sf w}\in \{\div, -, \times \}$, by $d^{\sf w}(f_1,f_2)$
we denote the number of operations ${\sf w}$ used in the
``synthetic'' polynomial division algorithm applied to $(f_1,f_2)$
(see, e.g., \cite{Knuth81}). It turns out that
\begin{equation}  \label{syntetic operations}
d^{\div}(f_1,f_2)=m-n+1, \quad d^{-, \times}(f_1,f_2)=n(m-n+1).
\end{equation}
Endowing $\fq[T]_d$ with the uniform probability, for any ${\sf
w}\in \{\div, -, \times \}$ we consider the random variable $t^{\sf
w}_g: \fq[T]_d \to \mathbb{N}$ which counts the number of operations
${\sf w}$ that the Euclidean Algorithm performs on input $(g,f)$ for
each $f\in\fq[T]_d$. Furthermore, $t^{\sf div}_g(f)$ denotes the
number of polynomial divisions involved. Our aim is to study the
expected value $E[t^{\sf w}_g]$ of $t^{\sf w}_g$ for ${\sf w}\in
\{\div,{\sf div}, -, \times \}$, namely
$$E[t^{\sf w}_g]=\frac{1}{q^d}\sum_{f\in \fq[T]_d} t^{\sf w}_g(f)
=\frac{1}{q^d}\sum_{k=0}^d\sum_{f\in B_{d-k}}t^{\sf w}_g(f).$$

As explained before, applying the Euclidean algorithm to an input
$(g,f)$ with $f\in \fq[T]_d$ we obtain a unique polynomial quotient
sequence $(q_1 \klk q_{h+1})$ and a unique polynomial remainder
sequence $(r_1 \klk r_h)$ satisfying the following conditions:
\begin{equation}\label{algorithm euclidean}\begin{array}{rclrcl}
 g&=&  f\cdot q_1+r_1, & \deg(r_1)&<&\deg(f),\\
 f&=&r_1\cdot q_2+r_2, & \deg(r_2)&<&\deg(r_1),\\
 &\vdots&&& \vdots \\
 r_{h-2}&=&r_{h-1}\cdot q_h+r_h,
 &\deg(r_h)&<&\deg(r_{h-1}),\\
 r_{h-1}&=&r_h\cdot q_{h+1}.
 \end{array}\end{equation}
We first consider ${\sf w}=\sf div$.
\begin{lemma}\label{lemma: E t div}
The average number $E[t_g^{\sf div}]$ of polynomial divisions
performed by the Euclidean algorithm applied to pairs $(g,f)$ with
$f\in\fq[T]_d$ is bounded as follows:
$$(d+1)\Bigg(1-\frac{d(2e-d+1)}{2q}\Bigg)\le
E[t_g^{\sf div}]\le (d+1)\Bigg(1+\frac{de}{q}\Bigg).$$
\end{lemma}
\begin{proof}
For $f \in B_{d-k}$ with $0\le k\le d$, we claim that $t_g^{\sf
div}(f) \leq k+1$. Indeed, the maximum number of polynomial
divisions in \eqref{algorithm euclidean} is achieved from a sequence
of remainders of maximum length. Since $f \in B_{d-k}$, in such a
sequence the degree of each successive remainder decreases by 1,
that is, the sequence has length $h=k$. Taking into account that
there is a further division to perform, to check that
$r_h$ divides $r_{h-1}$, we deduce our claim. 
As $k\mapsto \frac{k+1}{d-k}$ is an increasing function for $k \in
[0,d-1]$, we obtain
\begin{align*}
E[t^{\sf div}_g] \leq \frac{1}{q^d}\sum_{k=0}^d\sum_{f\in
B_{d-k}}(k+1)
&= \frac{1}{q^d} \Bigg( \sum_{k=0}^{d-1} \frac{k+1}{d-k} (d-k)|B_{d-k}|+(d+1)|B_0| \Bigg)\\
& \leq \frac{d}{q^d}\sum_{k=0}^{d-1} (d-k)|B_{d-k}| +(d+1)
\frac{|B_0|}{q^d} \le d\, E[\mathcal{X}_g]+ d+1.
\end{align*}
Using the bound $E[\mathcal{X}_g] \leq de/q$ of Lemma \ref{lemma:
average degree}, we deduce the upper bound in the statement of the
lemma.

Next we show the lower bound. Recall that $f\in\fq[T]_d$ is generic
(with respect to $g$) if the corresponding remainder sequence is of
the form $(r_1,\ldots,r_d)$, where $\deg(r_j)=d-j$ for $1\le j\le
d$. For such an $f$, the number of polynomial divisions is precisely
$d+1$. By Proposition \ref{prop: lower bound generic pols}, it
follows that
$$E[t^{\sf div}_g] \ge \frac{1}{q^d}(d+1)|\mathcal{G}|
\ge (d+1)\Bigg(1-\frac{d(2e-d+1)}{2q}\Bigg).$$
This finishes the proof of the lemma.
\end{proof}

Next we analyze the case ${\sf w}=\div$.
\begin{lemma}\label{lemma:E dividido}
Denote by $E[t_g^{\div}]$ the average number of divisions performed
by the Euclidean algorithm applied to pairs $(g,f)$ with
$f\in\fq[T]_d$. Then
$$(e+d+1)\Bigg(1-\frac{d(2e-d+1)}{2q}\Bigg)\le E[t_g^{\div}]
\le (e+d+1)\Bigg(1+\frac{de}{q}\Bigg).$$
\end{lemma}
\begin{proof}
Let $f\in B_{d-k}$ with $0\le k\le d$. According to \eqref{syntetic
operations}, the number of operations ${\div}$ in each step of
\eqref{algorithm euclidean} is
\begin{align*}
d^{\div}(g,f)&=\deg(g)-\deg(f)+1,\\
d^{\div}(f, r_1)&=\deg(f)-\deg(r_1)+1,\\
&\ \,\vdots  \\
d^{\div}(r_{h-1}, r_h), &=\deg(r_{h-1})-\deg(r_h)+1.
\end{align*}
Therefore,
\begin{equation}\label{t div k}
t_g^{\div}(f) =\deg(g)-\deg(r_h)+h+1 =e-(d-k)+h+1\le e-d+2k+1.
\end{equation}
%
%
As $k\mapsto\frac{e-d+2k+1}{d-k}$ is increasing for $k \in [0,d-1]$,
from \eqref{t div k} 
we deduce that
\begin{align*}
E[t^{\div}_g]= \frac{1}{q^d}\sum_{k=0}^d \sum_{f\in
B_{d-k}}t^{\div}_g(f) &\leq \frac{1}{q^d}  \sum_{k=0}^{d-1}
\frac{e-d+2k+1}{d-k}
(d-k)|B_{d-k}|+(e+d+1)\frac{|B_0|}{q^d}\notag\\
& \leq (e+d-1)E[\mathcal{X}_g]+e+d+1.
\end{align*}
Combining this with Lemma \ref{lemma: average degree} readily
implies the upper bound.

To prove the lower bound, we argue as in the proof of Lemma
\ref{lemma: E t div}. For a generic $f\in\mathcal{G}$, the remainder
sequence is of length $d$, and therefore $t_g^{\div}(f) =e+d+1$. It
follows that
$$E[t^{\div}_g] \ge \frac{1}{q^d}(e+d+1)|\mathcal{G}|
\ge (e+d+1)\Bigg(1-\frac{d(2e-d+1)}{2q}\Bigg).$$
This proves the lemma.
\end{proof}

Finally, we consider the remaining case ${\sf w} \in \{-, \times\}$.
We have the following result.
\begin{lemma}\label{lemma: E w}
Let $E[t^{-, \times}_g]$ be the average number of operations ${\sf
w}\in\{-,\times\}$ performed by the Euclidean algorithm applied to
pairs $(g,f)$ with $f\in\fq[T]_d$. Then
$$de\Bigg(1-\frac{d(2e-d+1)}{2q}\Bigg)\le E[t^{-, \times}_g]\le
de\Bigg(1+\frac{de}{q}\Bigg).$$
\end{lemma}
\begin{proof}
For $f\in B_{d-k}$ with $0\le k\le d$, by \eqref{syntetic
operations} the number of operations $d^{\sf w}$ with ${\sf
w}\in\{-,\times\}$ in each step of \eqref{algorithm euclidean} is
\begin{align*}
d^{\sf w}(g,f)&=\deg(f)(\deg(g)-\deg(f)+1),\\
d^{\sf w}(f, r_1)&=\deg(r_1)(\deg(f)-\deg(r_1)+1),\\
&\ \,\vdots  \\
d^{\sf w}(r_{h-1}, r_h)&=\deg(r_h)(\deg(r_{h-1})-\deg(r_h)+1) .
\end{align*}

Denote $r_0:=f$. We claim that the maximum number of operations
${\sf w}$ performed in the whole Euclidean algorithm is achieved
with a sequence of remainders $(r_0,\ldots,r_k)$ with
$\deg(r_{j-1})-\deg(r_j)=1$ for $1\le j\le k$. Indeed, let
$(r_0,\ldots,r_h)$ be a remainder sequence such that
$\deg(r_{j-1})-\deg(r_j)>1$ for a given $j$. Denote by
$(\alpha_0,\ldots,\alpha_h)$ the corresponding sequence of degrees.
We compare the number of operations ${\sf w}$ performed by the
Euclidean algorithm to obtain this sequence with that of a remainder
sequence
with degree pattern $(\alpha_0,\ldots,
\alpha_{j-1},\alpha_j^*,\alpha_j,\ldots,\alpha_h)$, where
$\alpha_{j-1}-\alpha_j^*=1$. Since the number of {\sf w} operations
is determined by the degree pattern of the remainder sequence under
consideration, it suffices to compare the cost of the $j$th step of
the first sequence with the sum of those of the $j$th and $(j+1)$th
steps of the second sequence. In particular, we see that our claim
for this case holds provided that
$$
\alpha_j\big((\alpha_{j-1}-\alpha_j)+1\big)\le
\alpha_j^*(\alpha_{j-1}-\alpha_j^*+1)+\alpha_j(\alpha_j^*-\alpha_j+1).
$$
This can be checked by an easy calculation. Arguing successively in
this way, the claim follows.

As a consequence, the maximum number of operations ${\sf w}$
performed is achieved in a sequence of $k$ remainders
$(r_1,\ldots,r_k)$ with $\deg(r_{j-1})-\deg(r_j)=1$ for $1\le j\le
k$, namely with $\deg(r_j)=d-j$ for $1\le j\le k$. It follows that
\begin{align}
t_g^{\sf w}(f)& \le
\deg(f)(\deg(g)-\deg(f)+1)+\sum_{j=1}^k\deg(r_j)(\deg(r_{j-1})-\deg(r_j)+1) \notag\\
&  = d(e-d+1)+2\sum_{j=1}^k(d-j)=d(e-d+1)+k(2d-k-1).\label{t w k}
\end{align}
%
%
%
Since $k\mapsto\frac{d(e-d+1)+k(2d-k-1)}{d-k} $ is increasing for $k
\in[0,d-1]$, by \eqref{t w k} 
we obtain
\begin{align*}
E[t^{\sf w}_g]&= \frac{1}{q^d}\sum_{k=0}^d\sum_{f\in B_{d-k}}t^{\sf
w}_g(f) \\
&\leq \frac{1}{q^d} \sum_{k=0}^{d-1} \frac{d(e-d+1)+k(2d-k-1)}{d-k} (d-k)|B_{d-k}|+ de\frac{|B_0|}{q^d}\\
& \leq de\, E[\mathcal{X}_g] +de.
\end{align*}
The upper bound follows easily by Theorem \ref{teo: average degree
mcd}.

On the other hand, for $f\in \mathcal{G}$, by \eqref{t w k} we
conclude that $t^{\sf w}_g(f)=de$. Then we have
$$E[t^{\div}_g] \ge \frac{1}{q^d}de|\mathcal{G}|
\ge de\Bigg(1-\frac{d(2e-d+1)}{2q}\Bigg),$$
which finishes the proof of the lemma.
\end{proof}

Summarizing Lemmas \ref{lemma: E t div}, \ref{lemma:E dividido} and
\ref{lemma: E w}, we have the following result.
\begin{theorem}\label{teo: average cost mcd}
Let $e, d$ be positive integers such that $q>d(2e-d+1)/2$ and $e>d$.
Let $g \in \fq[T]_e$ and ${\sf w} \in \{ \div, {\sf div}, -,
\times\}$. The average cost $E[t_g^{\sf w}]$ of operations ${\sf w}$
performed on (uniform distributed) inputs from $\fq[T]_d$ is bounded
in the following way:
$$
\bigg|\frac{E[t^{\sf div}_g]}{d+1}-1\bigg| \le\frac{de}{q},\qquad
\bigg|\frac{E[t^{\div}_g]}{e+d+1}-1\bigg|\le \frac{de}{q}, \qquad
\bigg|\frac{E[t^{-, \times}_g]}{de}-1\bigg|\le \frac{de}{q}.
$$
\end{theorem}
%
%
\section{Simulations on test examples}
\label{section: simulations}
In this section we report on the simulations we made with the
software package \texttt{Maple}. More precisely, for given values of
$q$, $e$ and $d$ with $e>d$, we executed the Euclidean algorithm on
pairs $(g,f)$, where $g \in \fq[T]_e$ was a fixed polynomial with
factorization pattern $(\lambda_1 \klk \lambda_e)\in
\mathbb{Z}_{\geq 0}^e$ and $f$ ran through all the elements of a
random sample $ \mathcal{S} \subset \fq[T]_d$. The aim was to
analyze to what extent the results of our simulations behaved as
predicted by the theoretical results on the average degree of
$\gcd(g,f)$ (Theorem \ref{teo: average degree mcd}), the probability
that $\gcd(g,f)=1$ (Theorem \ref{teo: prob coprime}) and the
probability that a random $f\in\fq[T]_d$ is ``generic'' with respect
to $g$ (Proposition \ref{prop: lower bound generic pols}).

Recall that, given $g\in \fq[T]_e$, we denote by $E[\mathcal{X}_g]$
the average degree of $\gcd(g,f)$ for $f$ running on all the
elements of $\fq[T]_d$. Further, the probability that $\gcd(g,f)=1$
is denoted by $\mathcal{P}_0$. Finally, we denote by
$\mathcal{P}_\mathcal{G}$ the probability that $f\in\fq[T]_d$ is
generic with respect to $g$. According to Theorems \ref{teo: average
degree mcd} and \ref{teo: prob coprime} and Proposition \ref{prop:
lower bound generic pols}, if the square-free part $g^*$ of $g$ has
factorization pattern $(\lambda_1^* \klk \lambda_e^*)$
and $k \leq d$ is the least index with
$\lambda_k^*>0$, then
$$E[\mathcal{X}_g] \approx {\tt E}_g:=\frac{k\lambda_k^*}{q^k},\quad
\mathcal{P}_0\approx {\tt P}_0:=1-\frac{\lambda_k^*}{q^k},\quad
\mathcal{P}_{\mathcal{G}}\geq {\tt P}_\mathcal{G}:=1 -
\frac{d(2e-d+1)}{2q}.
$$

The simulations we exhibit were aimed to test whether the right-hand
side in the previous expressions approximates the left-hand side on
the random samples under consideration.  For this purpose, given a
random sample $\mathcal{S} \subset \fq[T]_d$,
we computed the sample means
$$\mu:=\frac{1}{|\mathcal{S}|} \sum_{f \in \mathcal{S}} \deg\gcd(g,f),
\quad \beta:=\frac{|B_{0,s}|}{|\mathcal{S}|},\quad
\gamma:=\frac{|\mathcal{G}_s|}{|\mathcal{S}|},$$
where $B_{0,\mathcal{S}}:=\{f \in \mathcal{S}: \gcd(g,f)=1\}$ and
$\mathcal{G}_\mathcal{S}:=\{f \in \mathcal{S}: \,f \mbox{ is
generic} \}$.
Furthermore, we considered the corresponding relative errors
$$\varepsilon_1:=\frac{|\mu -{\tt E}_g|}{{\tt E}_g}, \quad \varepsilon_2=\frac{|\beta-{\tt P}_0|}{{\tt P}_0}.$$
%
%
\subsection{Examples for $q=67$, $e=7$, $d=3$ with $\lambda_1^* >0$}
Our first simulations concerned random samples $\mathcal{S}$ of
300000 polynomials $f\in\mathbb{F}_{67}[T]$ of degree at most $d:=3$
and polynomials $g\in\mathbb{F}_{67}[T] $ of degree $e:=7$ with
distinct values of $\lambda_1^*>0 $, listed in the first column of
Table \ref{table: 67, 7, 3}.
%
%
%
%
%
%
%
\newpage
\begin{center}
    \begin{table}[h]
        \caption{Examples with $q=67$, $e=7$ and $d=3$.}
        \label{table: 67, 7, 3}
        \begin{tabular}{|c|c|c|c|c|c|c|c|c|}
            \hline
            $\lambda_1^{*}$ & $\mu$ & ${\tt E}_g$ & $\beta$  & ${\tt P}_0$ &$\gamma$  & ${\tt P}_\mathcal{G}$  & $\varepsilon_1$ & $\varepsilon_2$ \\
            \hline
            $1$& $0.015273$ &$0.014925$  & $0.984963$ & $0.985075$  &$0.941560$  & $0.731343$ & $0.023312$& $0.000011$ \\
            \hline
            $2$& $0.030590$ & $0.029851$ & $0.969647$  & $0.970149$ & $0.927223$ & $0.731343$ & $0.024756$ & $0.000049$ \\
            \hline
            $3$& $0.045027$  & $0.044776$ & $0.956067$  & $0.955224$ &$0.914517$  & $0.731343$ & $0.000561$ & $0.000083$ \\
            \hline
            $4$& $0.059633$ & $0.059701$ & $0.942077$  & $0.940299$ &$0.900570$  & $0.731343$ & $0.000114$ &$0.000189$ \\
            \hline
            $5$& $0.074163$ & $0.074627$ & $0.928270$ & $0.925373$& $0.887490$  &$0.731343$ & $0.000622$ & $0.000313$ \\
            \hline
            $6$& $0.089130$  & $0.089552$ & $0.914273$  & $0.910448$ & $0.873877$  &$0.731343$ & $0.000471$&$0.000420$ \\
            \hline
            $7$& $0.103867$ & $0.104478$ & $0.900650$ & $0.895522$ & $0.860893$  &$0.731343$ & $0.000585$ & $0.000573$\\
            \hline
        \end{tabular}
    \end{table}
\end{center}

\subsection{Examples for $q=127$, $e=9$, $d=4$ with $\lambda_1^* >0$}
Next we considered random samples $\mathcal{S}$ of 10000000
polynomials $f\in\mathbb{F}_{127}[T]$ of degree at most $d:=4$. We
considered polynomials $g\in\mathbb{F}_{127}[T] $ of degree $e:=9$
with $\lambda_1^* >0 $ distinct roots. The corresponding results are
listed in Table \ref{table: 127, 9, 4}.

    \begin{center}
    \begin{table}[h]
        \caption{Examples with $q=127$, $e=9$ and $d=4$.}
        \label{table: 127, 9, 4}
        \begin{tabular}{|c|c|c|c|c|c|c|c|c|}
            \hline
            \!$\lambda_1^{*}$\! & $\mu$ & ${\tt E}_g$ & $\beta$  & ${\tt P}_0$ &$\gamma$  & ${\tt P}_\mathcal{G}$  & $\varepsilon_1$ & $\varepsilon_2$ \\
            \hline
            $1$& $0.008006$ & $0.007874$ & $0.992061$ & $0.992126$ & $0.961086$ & $0.763779$ & $0.024384$ & $0.000006$ \\
            \hline
            $2$& $0.015706$ & $0.015748$ & $0.984357$ & $0.984252$& $0.953689$  & $0.763779$ & $0.000267$ & $0.000011$ \\
            \hline
            $3$& $0.023653$ &$0.023622$ &$0.976539$  &$0.976378$ &$0.946121$ &$0.763779$ & $0.023591$ & $0.000016$ \\
            \hline
            $4$& $0.031474$  & $0.031496$ & $0.968891$  & $0.968504$ & $0.938736$ & $0.763779$& $0.000069$ & $0.000039$ \\
            \hline
            $5$& $0.039371$ & $0.039370$&$0.961239$  & $0.960629$ & $0.931346$ & $0.763779$ & $0.000003$ & $0.000064$ \\
            \hline
            $6$& $0.047185$ & $0.047244$ & $0.953743$ &$0.952756$ & $0.924084$ & $0.763779$ & $0.000125$ & $0.000104$ \\
            \hline
            $7$& $0.055216$ &$0.055118$ & $0.946135$  & $0.944882$& $0.916689$  & $0.763779$ & $0.000178$ & $0.000133$ \\
            \hline
            $8$& $0.062906$ & $0.062992$ & $0.938873$ & $0.937008$ & $0.909605$ & $0.763779$ & $0.000137$ & $0.000199$\\
            \hline
            $9$& $0.070742$ &$0.070866$ & $0.931443$  &$0.929133$ & $0.902504$ & $0.763779$ & $0.000175$ & $0.010884$ \\
            \hline
        \end{tabular}
    \end{table}
\end{center}

\subsection{Examples for $q=409$, $e=9$, $d=4$ with $\lambda_1^* >0$}
For our third family of examples we considered random samples
$\mathcal{S}$ of 10000000 polynomials $f\in\mathbb{F}_{409}[T]$ of
degree at most $d:=4$ and polynomials $g\in\mathbb{F}_{409}[T] $ of
degree $e:=9$ with $\lambda_1^* >0$. The corresponding results are
listed in Table \ref{table: 409, 9, 4}.
\newpage
\begin{center}
    \begin{table}[h]
        \caption{Examples with $q=409$, $e=9$ and $d=4$.}
        \label{table: 409, 9, 4}
        \begin{tabular}{|c|c|c|c|c|c|c|c|c|}
            \hline
            \!$\lambda_1^{*}$\! & $\mu$ & ${\tt E}_g$ & $\beta$  & ${\tt P}_0$ &$\gamma$  & ${\tt P}_\mathcal{G}$  & $\varepsilon_1$ & $\varepsilon_2$ \\
            \hline
            1& $0.002457$  &$0.002445$  &$0.997549$  &$0.997555$  & $0.987825$ & $0.926650$ & $0.000491$  & $0.000001$  \\
            \hline
            1& $0.002474$  & $0.002445$ &$0.997526$  &$0.997555$  & $0.987809$ &$0.926650$  & $0.011861$  & $0.000003$  \\
            \hline
            2& $0.004941$ & $0.004889$  &$0.995065$  &$0.995110$  &$0.985414$ & $0.926650$   &  $0.010636$& $0.044678$   \\
            \hline
            2& $0.004902$  & $0.004889$ & $0.995109$ & $0.995110$ & $0.985340$  &  $0.926650$  & $0.000268$  & $1 \cdot 10^{-6}$ \\
            \hline
            3& $0.007329$ &$0.007335$  & $0.992688$  &$0.992665$  & $0.983014$  & $0.926650$ &$0.000082$  & $0.000002$ \\
            \hline
            3& $0.007309$ &$0.007335$  &$0.992710$  &$0.992665$  & $0.983031$ & $0.926650$ &$0.00035$  & $0.000005$ \\
            \hline
            4& $0.009772$ & $0.009779$ & $0.99026$ &$0.990220$  &$0.980597$  & $0.926650$ &$0.000072$  & $0.00004$ \\
            \hline
            4& $0.009761$  &$0.009779$  &$0.990277$  & $0.990220$ & $0.980693$  & $0.926650$ &$0.000184$  & $0.000006$  \\
            \hline
            5& $0.012245$ &$0.012225$  &$0.987821$  &$0.987776$  & $0.978268$ & $0.926650$ &$0.000164$ & $0.002785$ \\
            \hline
            5& $0.012241$ &$0.012225$  & $0.987836$ & $0.987776$ & $0.978206$ & $0.926650$ &$0.000131$  & $0.000006$ \\
            \hline
            6& $0.014649$ &$0.014669$  &$0.985448$  &$0.985331$  & $0.975835$  & $0.926650$ &$0.000136$  & $0.000012$  \\
            \hline
            7& $0.017035$ & $0.017115$ & $0.983098$  & $0.982885$ & $0.973513$ & $0.926650$ &$0.000467$ &  $0.000022$ \\
            \hline
            8& $0.019524$ & $0.019552$  & $0.980654$  & $0.980440$  & $0.971072$  & $0.926650$ &$0.000143$  & $0.000022$ \\
            \hline
            9& $0.021948$ & $0.022005$ & $0.978267$ & $0.977995$  & $0.968726$  & $0.926650$  &$0.000278$ &$0.000028$  \\
            \hline
        \end{tabular}
    \end{table}

\end{center}
\subsection{Examples for $q=67$, $e=7$, $d=3$ with $\lambda_k^* >0$}
Now we report on random samples $\mathcal{S}$ of 300000 polynomials
$f\in\mathbb{F}_{67}[T]$ of degree at most $d:=3$, and polynomials
$g\in\mathbb{F}_{67}[T] $ of degree $e:=7$ having different values
for the least index $k$ with $\lambda_k^*>0$. In the first column of
Table \ref{table: 67, 7, 3, k>0} we show the different values of $k$
considered, while the second column exhibits the corresponding
values of $\lambda_k^*$. As the sample mean $\mu$ and the asymptotic
estimates ${\tt E}_g$ were close to zero, instead of the relative
error we considered the absolute error
$$\varepsilon_1:={|\mu - {\tt E}_g|},$$
listed in the ninth column of Table \ref{table: 67, 7, 3, k>0}.
%
%
\begin{center}
    \begin{table}[h]
        \caption{Examples with $q=67$, $e=7$ and $d=3$.}
        \label{table: 67, 7, 3, k>0}
        \begin{tabular}{|c|c|c|c|c|c|c|c|c|c|}
            \hline
            $k$ & \!$\lambda_k^{*}$\!  & $\mu$ & ${\tt E}_g$ & $\beta$  & ${\tt P}_0$ &$\gamma$  & ${\tt P}_\mathcal{G}$  & $\varepsilon_1$ & $\varepsilon_2$ \\
            \hline
            2& $1$ & $0.000453$ &$0.000445$  & $0.999773$ & $0.999777$  & $0.955343$  & $0.731343$ & $0.000008$ & $0.000004$\\
            \hline
            2& $2$ & $0.000836$ & $0.000891$ & $0.999586$& $0.999555$  &$0.0955466$ & $0.731343$& $0.000055$ & $0.000003$ \\
            \hline
            3& $1$ & $0.000030$ & $0.000001$ &$0.999990$ &  $0.999997$ &$0.955866$ & $0.731343$ & $0.000029$ &$0.000007$ \\
            \hline
            3& $1$ & $0.000030$ & $0.000001$ & $0.999990$ & $0.999997$ &$0.956393$ & $0.731343$ & $0.000029$ & $0.000009$ \\
            \hline

        \end{tabular}
    \end{table}
\end{center}

\subsection{Examples for $q=127$, $e=9$, $d=4$ with $\lambda_k^* >0$}
The next family of examples concerned random samples $\mathcal{S}$
of 10000000 polynomials $f\in\mathbb{F}_{127}[T]$ of degree at most
$d:=3$. We considered polynomials $g\in\mathbb{F}_{127}[T] $ of
degree $e:=9$ with different values for the least index $k$ with
$\lambda_k^*>0$. The corresponding results are summarized in Table
\ref{table: 127, 9, 4, k>0}. As the sample mean $\mu$ and the
asymptotic estimates ${\tt E}_g$ were close to zero, we considered
the absolute error
$$\varepsilon_1:={|\mu - {\tt E}_g|}.$$

\begin{center}
    \begin{table}[h]
        \caption{Examples with $q=127$, $e=9$ and $d=4$.}
        \label{table: 127, 9, 4, k>0}
        \begin{tabular}{|c|c|c|c|c|c|c|c|c|c|}
            \hline
            $k$ & \!$\lambda_k^{*}$\!   & $\mu$ & ${\tt E}_g$ & $\beta$  & ${\tt P}_0$ &$\gamma$  & ${\tt P}_\mathcal{G}$  & $\varepsilon_1$ & $\varepsilon_2$ \\
            \hline
            2& $1$ &$0.000135$ & $0.000124$& $0.999932$ & $0.999938$ & $0.968791$ & $0.763779$ & $0.000011$ & $0.000006$ \\
            \hline
            2& $2$ & $0.000262$
            &$0.000248$ & $0.999869$ & $0.999876$ &$0.968705$ & $0.763779$ & $0.000014$ & $0.000007$ \\
            \hline
            3& $1$& $0.000006$ & $0.000001$ & $0.999999$ & $0.999999$ & $0.968915$ & $0.763779$& $0.000005$ & $0$\\
            \hline
            4& $1$ & $0$ & $2 \cdot 10^{-9}$  & $1$ & $0.999999$ & $0.968926$ &$0.763779$ & $2\cdot 10^{-9}$ & $ 0.000001$ \\
            \hline
        \end{tabular}
    \end{table}
\end{center}

\subsection{Examples for $q=211$, $e=17$, $d=7$ with $\lambda_k^* >0$}
Now we report on simulations with random samples $\mathcal{S}$ of
10000000 polynomials $f\in\mathbb{F}_{211}[T]$ of degree at most
$d:=3$.  We considered polynomials $g\in\mathbb{F}_{211}[T] $ of
degree $e:=17$ having different values for the least index $k$ with
$\lambda_k^*>0$. As the sample mean $\mu$ and the asymptotic
estimates ${\tt E}_g$ were close to zero, we considered the absolute
error
$$\varepsilon_1:={|\mu - {\tt E}_g|}.$$
The results are listed in Table \ref{table: 211, 17, 7}.
%

\begin{center}
    \begin{table}[h]
        \caption{Examples with $q=211$, $e=17$ and $d=7$.}
        \label{table: 211, 17, 7}
        \begin{tabular}{|c|c|c|c|c|c|c|c|c|c|}
            \hline
            $k$ & \!$\lambda_k^{*}$\!   & $\mu$ & ${\tt E}_g$ & $\beta$  & ${\tt P}_0$ &$\gamma$  & ${\tt P}_\mathcal{G}$  & $\varepsilon_1$ & $\varepsilon_2$ \\
            \hline
            2& $2$ &$0.000087$ & $0.000089$  & $0.999956$ & $0.999955$  & $0.967239$ & $0.535545$ & $0.000002$& $ 0.000001$ \\
            \hline
            2& $1$ &$0.000042$ & $0.000045$& $0.999979$ & $0.999978$& $0.967284$ & $0.535545$& $0.000003$& $0.000001$ \\
            \hline
            3&$1$ & $0$ & $3 \cdot 10^{-7}$& $1$ & $0.999999$ & $0.967244$ &$0.535545$ & $3 \cdot 10^{-7}$  & $ 0.000001$ \\
            \hline
            3& $2$ & $0$&  $6 \cdot 10^{-7}$& $1$& $0.999999$ & $0.967236$ &$0.535545$ & $6 \cdot 10^{-7}$& $ 0.000001$\\
            \hline
            %

        \end{tabular}
    \end{table}
\end{center}

\subsection{Examples for $q=409$, $e=9$, $d=4$ with $\lambda_k^* >0$}
The last family of examples involves random samples $\mathcal{S}$ of
10000000 polynomials $f\in\mathbb{F}_{409}[T]$ of degree at most
$d:=4$. We considered polynomials $g\in\mathbb{F}_{409}[T]$ of
degree $e:=9$ having different values for the least index $k$ with
$\lambda_k^*>0 $. As the sample mean $\mu$ and the asymptotic
estimates ${\tt E}_g$ were close to zero, we considered the absolute
error
$$\varepsilon_1:={|\mu - {\tt E}_g|}.$$
The results are exhibited in Table \ref{table: 409, 9, 4, k>0}.

%
%
%
%
%
%
%
%

\begin{center}
    \begin{table}[h]
        \caption{Examples with $q=409$, $e=9$ and $d=4$.}
        \label{table: 409, 9, 4, k>0}
        \begin{tabular}{|c|c|c|c|c|c|c|c|c|c|}
            \hline
            $k$ &\!$\lambda_k^{*}$\!  & $\mu$ & ${\tt E}_g$ & $\beta$  & ${\tt P}_0$ &$\gamma$  & ${\tt P}_\mathcal{G}$  & $\varepsilon_1$ & $\varepsilon_2$ \\
            \hline
            2& $2$ & $0.000026$ & $0.000024$ & $0.999987$ & $0.999988$ & $0.931219$&$0.926650$ & $0.000002$& $0.000001$ \\
            \hline
            2&  1& $0.000015$ & $0.000002$& $0.999993$&  $0.999994$& $0.990212$&$0.926650$ & $0.000013$& $0.000001$ \\
            \hline
            3& $2$ & $0$ & $8 \cdot 10^{-8}$ & $1$ & $0.999999$ & $0.946044$ &$0.926650$ & $8 \cdot 10^{-8}$ & $0.000001$ \\
            \hline
            3& $1$ & $0$ & $ 4 \cdot 10^{-8}$ & $1$ & $0.999999$ & $0.946146$ &$0.926650$ & $ 4 \cdot 10^{-8}$ & $0.000001$ \\
            \hline

        \end{tabular}

    \end{table}

\end{center}

\subsection{Conclusions}
Summarizing, the results of Tables \ref{table: 67, 7,
3}--\ref{table: 409, 9, 4, k>0} show that the numerical experiments
we performed behave as predicted by the asymptotic estimates of
Theorems \ref{teo: average degree mcd} and \ref{teo: prob coprime}.
On the other hand, it seems that the estimate on the number generic
polynomials of Proposition \ref{prop: lower bound generic pols} is
somewhat pessimistic. Our numerical experiments suggest that the
number of generic polynomials depends on the factorization pattern
of $g$, while the lower bound of Proposition \ref{prop: lower bound
generic pols} depends only on $q$, $e$ and $d$.

\bibliographystyle{amsalpha}

\bibliography{refs1,finite_fields,polyhedral}

\end{document}